\documentclass[11pt,reqno]{amsart}

\usepackage{a4wide}
\usepackage{amsmath,amssymb,amsfonts,mathrsfs,verbatim,enumitem,pstricks,amsthm, graphicx,calligra}
\usepackage[all]{xy}
\usepackage{centernot, xr, accents}
\usepackage{hyperref}

\usepackage{amsmath,amssymb,mathtools}
\usepackage[all]{xy}
\usepackage{hyperref}
\usepackage{verbatim}

\usepackage[bottom]{footmisc}
\usepackage{lipsum,cleveref}

\usepackage{hyperref}
\hypersetup{
	colorlinks,
	citecolor=red,
	filecolor=black,
	linkcolor=blue,
	urlcolor=black
}

\DeclareMathOperator{\Hom}{\mathscr{H}\text{\kern -3pt {\calligra\Large om}}\,}

\theoremstyle{plain}

\newtheorem{thm}{Theorem}[section]

\newtheorem{lmm}[thm]{Lemma}
\newtheorem{prp}[thm]{Proposition}

\theoremstyle{definition}

\newtheorem{rem}{Remark}

\theoremstyle{definition}

\numberwithin{equation}{section}

\newcommand{\cO}{\mathcal{O}}

\def \hf{\hspace*{0.5cm}}

\begin{document}
\title[torsion-freeness of K\"{a}hler differential sheaves]{On torsion-freeness of K\"{a}hler differential sheaves}
\author[N. Das]{Nilkantha Das}

\address{Stat-Math Unit, Indian Statistical Institute, 203 B.T. Road, Kolkata 700 108, India.}
\email{dasnilkantha17@gmail.com}

\author[S. Roy]{Sumit Roy}
\address{Stat-Math Unit, Indian Statistical Institute, 203 B.T. Road, Kolkata 700 108, India.}
\email{sumit.roy061@gmail.com}
\subjclass[2020]{ 14F10, 13N05, 14B05}
\keywords{K\"{a}hler differential, torsion-free, normal variety}

\begin{abstract}
Let $X$ be a normal algebraic variety over an algebraically closed field $k$ of characteristic zero. We prove that the K\"{a}hler differential sheaf of $X$ is torsion-free if and only if any regular section of the ideal sheaf of the first order deformation of $X$ inside $X\times_k X$, defined outside the singular locus of $X \times_k X$, extends regularly to the singular locus.
\end{abstract}

\maketitle

\section{Introduction}\label{intro} 
Let $k$ be an algebraically closed field of characteristic zero, and $X$ be an algebraic variety over $k$. Among other few objects that come naturally with the structure of $X$, the sheaf of K\"{a}hler differentials $\Omega_X:= \Omega_{X/k}$ and its higher exterior powers are the central ones. For example, the algebraic de Rahm cohomology theory can be defined using the complex of regular differential forms (cf. \cite{Hartshorne_Algebraic}). There are several open problems that predict the structure of a variety in terms of the sheaf of K\"{a}hler differentials. For example, the Berger conjecture: a curve is non-singular if and only if its K\"{a}hler differential sheaf is torsion-free (\cite{Berger, Huneke}), and the Zariski-Lipman conjecture: a normal variety $X$ is non-singular if and only if the dual of $\Omega_X$ is locally free. The Zariski-Lipaman conjecture has been studied extensively by several mathematicians over the last few decades (cf. \cite{ Biswas-Gurjar-Kolte, Flenner,Graf,Graf-Kovacs, Jorder, Lipman, Miller-Vassiliadou}). The study of the various properties of $\Omega_X$ is an active area of research till date. \\
\hf It is well known that an algebraic variety $X$ is non-singular if and only if $\Omega_X$ is locally free. It is also well known that the torsion-free sheaves form a larger class than the class of locally free sheaves, and the reflexive sheaves form a larger class than the torsion-free ones. Therefore for singular varieties, it is quite natural to ask for conditions on $X$ for which $\Omega_X$ is reflexive or at least torsion-free, or if there is an equivalent criterion for $\Omega_X$ to be  torsion-free or reflexive. As mentioned in \cite{Red_book}, these kind of questions are quite interesting. Reflexivity of differential forms has some interesting geometrical consequences (for more details, see \cite{Greb-KP, Kebekus}). Lipman in \cite[Proposition 8.1]{Lipman} proved that if $X$ is a normal and local complete intersection, then $\Omega_X$ is torsion-free, and it is reflexive if and only if $X$ is regular in codimension $2$. An alternative proof of this result can be found in  \cite[Corolloary 9.8]{Kunz}. In general, $\Omega_X$ is not torsion-free, even if we consider $X$ to be normal. For example, the K\"{a}hler differential sheaf of the affine cone over the twisted cubic has torsion elements (cf. \cite{Greb-R}). In \cite{Greb-R}, Greb and Rollenske developed an equivalent criterion for the differential sheaf of the affine cone over a smooth projective variety to be torsion-free. They proved that the differential sheaf of the affine cone over a smooth projectively normal variety is torsion-free if and only if the first infinitesimal neighbourhood of the variety is projectively normal. Lipman in \cite[Lemma, p. 897]{Lipman} also gave an equivalent cohomological criterion for a coherent sheaf to be torsion-free and reflexive, respectively. It is worthwhile to note that the above criterion works for any coherent sheaf. Nothing special happens if we restrict our attention to $\Omega_X$. It is tempting to ask for special equivalent criterion for $\Omega_X$ to be torsion-free, or reflexive. \\
\hf The goal of the present note is to give an equivalent criterion for $\Omega_X$ to be torsion-free when $X$ is a normal variety. We have proved the following
\begin{thm}\label{torsion-free_theorem}
Let $X$ be a normal algebraic variety over an algebraically closed field $k$ of characteristic zero. Then the K\"{a}hler differential sheaf $\Omega_{X/k}$ is torsion-free if and only if the natural morphism $$\mathcal{I}^2 \longrightarrow j_*j^* \mathcal{I}^2$$ is surjective, where $\mathcal{I}$ is the ideal sheaf given by the diagonal embedding $\Delta:X \longrightarrow X \times_k X$ and $j$ is the inclusion map of the non-singular locus of $X \times_k X$ inside $X \times_k X$.
\end{thm}
\noindent It is worthwhile to note that $\mathcal{I}^2$ is the ideal sheaf of the first order deformation of the image of $X$ inside $X \times_k X$ under the diagonal morphism $\Delta$. \Cref{torsion-free_theorem} can be rephrased as follows:
the K\"{a}hler differential sheaf $\Omega_{X/k}$ is torsion-free if and only if for an open subset $U$ of $X \times_k X$, a section of $\mathcal{I}^2$ defined outside the singular locus of $U$, can be extended to a section of it over $U$. 
Loosely speaking, torsion-freeness of $\Omega_X$ is equivalent to an extension problem of certain sections defined outside the singular locus of $X \times_k X$.\\
\hf Let us phrase the problem slightly differently. The torsion submodule of $\Omega_{X}$ is precisely the kernel of the natural map 
\begin{equation*}
\Omega_{X}\longrightarrow \Omega_{X}^{\vee \vee},
\end{equation*}
where $ \Omega_{X}^{\vee \vee}$ is the double dual of $ \Omega_{X}$. Thus $\Omega_{X}$ is torsion-free if and only if the above morphism is injective. Therefore \Cref{torsion-free_theorem} is now converted to prove the injectivity of the above morphism. This observation will be used in the proof of \Cref{torsion-free_theorem}.\\

\noindent\textbf{Notation:} Throughout this note we assume $k$ is an algebraically closed field of characteristic zero. An algebraic variety over $k$ stands for an integral separated scheme of finite type over $k$. Given an algebraic variety $X/k$, the K\"{a}hler differential sheaf $\Omega_{X/k}$ is denoted by $\Omega_X$ for the notational convenience. Given an algebraic variety $X$, the diagonal map $X \longrightarrow X \times_k X$ is denoted by $\Delta_X$ (or simply by $\Delta$ if there is no cause of confusion).
\section{Preliminaries}
In this section, we recall some basic facts that we need while proving \Cref{torsion-free_theorem}. Although the following result is well-known, we will sketch a proof for the convenience of the reader.
\begin{lmm}\label{normality preserves under product}
Let $k$ be an algebraically closed field, and $X$, $Y$ be two algebraic varieties over $k$. If both $X$ and $Y$ are normal, so is their product $X \times_k Y$.  
\end{lmm}
\begin{proof}
    Recall that a morphism of schemes is said to be normal if it is flat and for every fiber of the morphism is geometrically normal. Since $k$ is algebraically closed, $X$ is a normal variety if and only if $X$ is geometrically normal over $k$. Also $X \rightarrow \text{spec }k$ is always flat. Therefore, $X$ is a normal variety if and only if $X \rightarrow \text{spec }k$ is normal. By the normality theorem (cf. \cite{EGA-IV}, \cite[Lemma 33.10.5]{stacks-project}), we conclude that $X\times_k Y$ is a normal scheme over $k$. We are done.
\end{proof}
As a consequence of the above result, we get if $X$ is a normal variety over $k$, then so is $X \times_k X$. \\
\hf Next, we move towards an extension problem of sections of certain kind of coherent ideal sheaves of a normal scheme. It is well-known that a reflexive sheaf $\mathcal{F}$ on a normal scheme $X$ is normal. That is, given an open subset $V$ of $X$ and an open subset $U\subseteq V$ that contain all the codimension $1$ points of $V$, any section of $\mathcal{F}$ defined over $U$ can be extended to a section over $V$. Though the ideal sheaves are not reflexive in general, they enjoy the same extension property for some special kind of open subsets $U$.  
\begin{prp} \label{ideal_sheaf_isomorphism_lemma}
Let $\mathcal{I}$ be an ideal sheaf of a normal integral locally Noetherian scheme $X$ over $k$ such that it defines an integrally closed subscheme $Y$ of $X$. Assume $U$ is an open subset of $X$ containing all the codimension $1$ points of $X$ and $U \cap Y \neq \emptyset$. If the inclusion of $U$ inside $X$ is denoted by $j$, the  the natural morphism 
    \begin{equation}\label{ideal_morphism}
    \mathcal{I} \longrightarrow j_*j^*\mathcal{I}    
    \end{equation}
   of sheaves on $X$ is an isomorphism.
\end{prp}
\begin{proof}
If $X$ is an algebraic curve, the result is straightforward. Form now onwards, we assume $\text{dim }X \geq 2$. First, we will show that the map in \cref{ideal_morphism} is injective. Consider the following commutative diagram
   \begin{equation*}
\begin{gathered}
\xymatrix{
 \mathcal{I} \ar[rr] \ar[d] & & j_*j^*\mathcal{I} \ar[d]\\
 \cO_X \ar[rr]_{\simeq} & & j_*j^* \cO_X
}
\end{gathered} 
\end{equation*}
Note that $X$ is assumed to be normal and codimension of $X \setminus U$ is $\geq 2$, Hartog's phenomenon immediately yields that the natural map $$\cO_X \longrightarrow j_*\cO_U$$ is an isomorphism. That is, the lower horizontal morphism in the diagram is an isomorphism. Since the vertical arrows are injective, it follows that the morphism in \cref{ideal_morphism} is injective.\\
Surjectivity of \cref{ideal_morphism} is quite subtle. Observe that $\mathcal{I}_U$ (equivalently, $j^*\mathcal{I}$) defines a closed subscheme of $U$ whose support is $V:=U \cap Y$. Next we will show that $j_*j^* \mathcal{I}$ defines a closed subscheme of $X$ whose support is $Y$, and subsequently, we show that $\mathcal{I}= j_*j^* \mathcal{I}$. Observe that $\left( V, \cO_U/\mathcal{I}_U \right)$ is a locally closed subscheme of $X$. Now applying \cite[Proposition 2.3.11]{Mumford-Oda}, we conclude that the ideal sheaf $\overline{\mathcal{I}}$ of the scheme theoretic closure of $V$ is given by
$$\overline{\mathcal{I}}(W)= \ker \left( \cO_{X}(W) \longrightarrow \cO_V(V \cap W) \right).$$
Now Hartog's phenomenon immediately yields that $\cO_X(W)= \cO_X(W \cap U)$. On the other hand, $\cO_V(V \cap W)$ is the same as $\left.\dfrac{\cO_X}{I}\right|_U(Y \cap W \cap U)$ which is again the same as $\cO_Y(W \cap U)$. Thus the above kernel is the same as $\left(j_*j^*\mathcal{I}\right)(W)$. Therefore, we conclude that $\overline{\mathcal{I}}=j_*j^*\mathcal{I}$. On the other hand, observe that the underlying topological space of the scheme theoretic closure of $V$ is the closure of $V$ in $X$. Since $Y$ is irreducible, $V$ is a non-empty open subset of $Y$, and hence, dense in $Y$. It follows that the closure of $V$ in $X$ is $Y$ itself. Thus we conclude that $j_*j^* \mathcal{I}$ defines a closed subscheme of $X$ whose support is $Y$. Now observe that $Y$ is assumed to be a reduced closed subscheme. By the uniqueness of reduced closed subscheme structure with a given support (cf. \cite[Proposition 2.3.11]{Mumford-Oda}), it follows that $j_*j^*\mathcal{I} \subseteq \mathcal{I}$. Since the morphism in \cref{ideal_morphism} is injective, it follows that $\mathcal{I}= j_*j^* \mathcal{I}$. This completes the proof.
\end{proof}
\begin{rem}
    In light of criterion (iii) of \cite[Proposition 1.6]{Hartshorne_stable_reflexive_sheaves}, one might think that $\mathcal{I}$ is a reflexive $\mathcal{O}_X$-module. But this is not the case. In fact, criterion (iii) of \cite[Proposition 1.6]{Hartshorne_stable_reflexive_sheaves} is demanding something more than that of  \Cref{ideal_sheaf_isomorphism_lemma}. 
    Indeed, if we consider $Y$ is of codimension $\geq 2$, we can't apply \Cref{ideal_sheaf_isomorphism_lemma} for $U= X \setminus Y$ as $U \cap Y= \emptyset$. This is a place where criterion (iii) of \cite[Proposition 1.6]{Hartshorne_stable_reflexive_sheaves} may fails.
\end{rem}
\begin{rem}
    It is worthwhile to note that the hypothesis $U \cap Y \neq \emptyset$ is important. If we remove this hypothesis, \Cref{ideal_sheaf_isomorphism_lemma} is not true anymore. For example,  consider $Y$ is of codimension $\geq 2$, and $U= X \setminus Y$. Then $j_*j^* \mathcal{I}$ is the same as $j_*\cO_U$. The latter is the same as $\cO_X$, thanks to the Hartog's phenomenon. Therefore, the morphism in \cref{ideal_morphism} becomes the natural injection of the ideal sheaf $\mathcal{I}$ inside $\cO_X$, but not an isomorphism. In fact, if we drop this hypothesis from \Cref{ideal_sheaf_isomorphism_lemma},
    it is immediate to conclude that $\mathcal{I}$ is reflexive. Now applying \cite[Corollary 1.5]{Hartshorne_stable_reflexive_sheaves} over the short exact sequence $$0 \longrightarrow \mathcal{I} \longrightarrow \cO_X \longrightarrow \cO_X/\mathcal{I}\longrightarrow 0,$$ we conclude that $\text{Ass }Y$ consists points of codimension $0$ and $1$ in $X$. Thus we get a restriction on $Y$.
\end{rem}
\section{Proof of the theorem}
We have already noticed that given an algebraic variety $X$, $\Omega_X$ is torsion-free if and only if the natural induced map 
\begin{equation}\label{double_dual_natural_map_2}
    \Omega_X \longrightarrow \Omega_X^{\vee \vee}
\end{equation}
of sheaves is injective.
We will first find an alternative criterion of injectivity of the map in \cref{double_dual_natural_map_2}. The following result will be useful in our purpose.
\begin{prp}
Let $U$ be the set of all non-singular points of a normal algebraic variety $X$ over $k$. Let us denote the natural inclusion of $U$ in $X$ by $i$. Then the morphism in \cref{double_dual_natural_map_2} is injective if and only if the natural morphism of sheaves of $\cO_X$-modules
\begin{equation}\label{pull-push-equation}
\Omega_X \longrightarrow i_* i^* \Omega_X
\end{equation}
 is injective. 
\end{prp} 
\begin{proof}
If $\text{dim }X=1$, the result is immediate. Thus we assume $\text{dim }X \geq 2$. 
Let us denote $S$ to be the space of singular points of $X$. Then codimension of $S$ is at least $2$ in $X$. Consider the following commutative diagram:
\begin{equation}\label{sheaf_diagram}
\begin{gathered}
\xymatrix{
 \Omega_{X} \ar[r] \ar[d] &  \Omega_{X}^{\vee \vee} \ar[d]\\
 i_* i^*  \Omega_{X} \ar[r] &  i_*i^* \Omega_{X}^{\vee \vee}
}
\end{gathered} 
\end{equation}
\hf First note that $i^*\Omega_{X}=\Omega_U$. Now observe that $i^*\Omega_X^{\vee \vee}$ is the same as $\Omega_U^{\vee \vee}$. This follows from the general fact that given two sheaves $\mathcal{F}$ and $\mathcal{G}$ of $\cO_X$-modules, the sheaf $i^*\Hom_{\cO_X} (\mathcal{F}, \mathcal{G})$ is the same as $\Hom_{\cO_U} (i^*\mathcal{F}, i^*\mathcal{G})$. Now $U$ being the non-singular locus of $X$, the sheaf $\Omega_U$ is locally free, and hence the natural map $$\Omega_U \longrightarrow \Omega_{U}^{\vee \vee}$$ is an isomorphism. The functionality of $i_*$ immediately yields that the lower horizontal arrow of the diagram \eqref{sheaf_diagram} is an isomorphism.  \\
\hf Next we will show that the right vertical arrow of the diagram is an isomorphism. Note that $\Omega_X^{\vee}$ is coherent, and thus $\Omega_X^{\vee \vee}$ is reflexive by \cite[Corollary 1.2]{Hartshorne_stable_reflexive_sheaves}. Now apply \cite[Proposition 1.6]{Hartshorne_stable_reflexive_sheaves} to conclude the isomorphism.\\
\hf We get both the lower horizontal arrow and the right vertical arrow are isomorphisms. The commutativity of the diagram \eqref{sheaf_diagram} gives us that injectivity of \eqref{double_dual_natural_map_2} and \eqref{pull-push-equation} are equivalent.
\end{proof}
Now we are in a position to prove our main result. For the convenience of the reader we repeat the full statement.
\begin{thm}\label{injectivity of sheaf of differential}
Let $X$ be a normal algebraic variety over $k$, and $Y$ be the singular locus of $X$. If $i$ denotes the inclusion map of $U := X \setminus Y$ in $X$, then the natural morphism $\Omega_{X} \longrightarrow  i_* i^*  \Omega_{X}$ is injective if and only if the natural map $$\mathcal{I}^2 \longrightarrow j_*j^* \mathcal{I}^2$$ is surjective, where $j$ is the inclusion of $U \times_k U$ in $X\times_k X$.
\end{thm}
\begin{proof}
If $X$ is a normal algebraic curve, it is always non-singular. Thus $U=X$, and both $i, j$ are the identity maps. The theorem follows immediately.\\
\hf From now onward we will assume $\dim X \geq 2$.
Let $$\Delta_X: X \longrightarrow X \times_k X$$ be the diagonal morphism. Then we get the following commutative diagram 
\begin{equation}\label{commutative diagram}
\begin{gathered}
\xymatrix{
U \ar[rr]^{\Delta_U} \ar[d]_i && U \times_k U \ar[d]^{i \times_k i} \\
X \ar[rr]^{\Delta_X} && X \times_k X
}
\end{gathered}
\end{equation}

Let us denote the map $i \times_k i$ by $j$ for the sake of notational simplicity. As $X$ is separated, $\Delta_X(X)$ (respectively $\Delta_U(U)$) is closed in $X \times_k X$ ( $ U \times_k U$, respectively). Note that $X \times_k X$ is also a normal variety by \Cref{normality preserves under product}. Let $\mathcal{I}$ be the ideal sheaf of $\Delta_X(X)$ in $X \times_k X$. Then $j^* \mathcal{I}$ will be the ideal sheaf of $\Delta_U(U)$ in $U \times_k U$. \\
\hf Note that $\left( X \times_k X\right)$ is a normal variety, and a straightforward dimension calculation shows that the codimension
\[
\mathrm{codim }_{X \times_k X} \big(\left( X \times_k X\right)\setminus \left( U \times_k U \right)\big) \geq 2.
\]
Thus the natural morphism $$\cO_{X \times_k X} \longrightarrow j_*j^* \cO_{X \times_k X}$$ is an isomorphism, thanks to the Hartog's phenomenon.  On the other hand, $\Delta_X(X)$, being isomorphic to $X$, is a reduced and irreducible closed subscheme of $X \times_k X$. Also $\Delta_X(X) \cap \left( U \times_k U \right)$ is the same as $\Delta_U(U)$, which is non-empty. It follows from \Cref{ideal_sheaf_isomorphism_lemma} that  the natural morphism of ideals $$\mathcal{I} \longrightarrow j_*j^* \mathcal{I}$$ is an isomorphism.\\
\hf Now consider the following commutative diagram:
   \begin{equation*}
\begin{gathered}
\xymatrix{
 \mathcal{I}^2 \ar[rr] \ar[d] & & j_*j^*\mathcal{I}^2 \ar[d]\\
 \cO_{X \times_k X} \ar[rr]_{\simeq} & & j_*j^* \cO_{X \times_k X}
}
\end{gathered} 
\end{equation*}
Due to Hartog's phenomenon, the lower horizontal arrow is an isomorphism. Being ideal of the respective sheaves, both the vertical arrows are injective as well. It follows that the map 
$$\mathcal{I}^2 \longrightarrow j_*j^* \mathcal{I}^2$$ is always injective.\\
\hf Note that as $\cO_{X \times_k X}$-module the sheaf $\Omega_{X}$ is the same as $\mathcal{I}/\mathcal{I}^2$. Similarly, as $\cO_{U \times_k U}$-module the sheaf $\Omega_{U}$ is the same as $j^*\mathcal{I}/j^*\mathcal{I}^2$. Now let us  analyse the sheaf $i_*i^*\Omega_{X }$, i.e., $i_*\Omega_{U}$. The pushforward sheaf $i_*\Omega_{U /k}$ is isomorphic to  $j_* \left( j^* \mathcal{I}/ j^* \mathcal{I}^2 \right) $ as $\cO_{X \times_k X}$-module . Consider the following short exact sequence of $\cO_{U \times_k U}$-modules:
$$0 \longrightarrow j^* \mathcal{I}^2 \longrightarrow j^* \mathcal{I} \longrightarrow  j^* \mathcal{I}/ j^* \mathcal{I}^2 \longrightarrow 0.$$ Applying $j_*$ we get the following exact sequence of $\cO_{X \times_k X}$-modules:
$$0 \longrightarrow j_*j^* \mathcal{I}^2 \longrightarrow j_*j^* \mathcal{I} \longrightarrow j_* \left( j^* \mathcal{I}/ j^* \mathcal{I}^2 \right). $$
Now consider the following diagram of exact sequences of $\mathcal{O}_{X \times_k X}$-modules:

\begin{equation}\label{exact sequence diagram}
\begin{gathered}
\xymatrix{
0 \ar[r] & \mathcal{I}^2 \ar[r] \ar[d] & \mathcal{I} \ar[r] \ar[d] & \mathcal{I}/ \mathcal{I}^2 \ar[r] \ar[d] & 0 \\
0 \ar[r] & j_*j^* \mathcal{I}^2 \ar[r] & j_*j^* \mathcal{I} \ar[r] & j_* \left( j^* \mathcal{I}/ j^* \mathcal{I}^2 \right) &
}
\end{gathered}
\end{equation}

Note that the middle vertical arrow is an isomorphism and the left vertical one is injective. Applying the snake lemma, we conclude that the morphism $$\mathcal{I}/ \mathcal{I}^2  \longrightarrow j_* \left( j^* \mathcal{I}/ j^* \mathcal{I}^2 \right)$$ is injective if and only if $$\mathcal{I}^2 \longrightarrow j_*j^* \mathcal{I}^2$$ is an isomorphism. On the other hand, $$\Omega_{X} \longrightarrow  i_* i^*  \Omega_{X}$$ is injective if and only if $$\left(\Delta_X\right)_* \Omega_{X} \longrightarrow  \left(\Delta_X\right)_* i_*  \Omega_{U}$$ is injective as well. And the latter morphism is nothing but the rightmost vertical arrow in \cref{exact sequence diagram} as follows from the commutativity of the diagram in \cref{commutative diagram}. This completes the proof.
\end{proof}
\begin{rem}
    In the above proof, we note that the morphism $$\mathcal{I}^2 \longrightarrow j_*j^* \mathcal{I}^2$$ between the ideal sheaves is always injective and we also note that this morphism is not always an isomorphism. Therefore the \Cref{ideal_sheaf_isomorphism_lemma} can't be applicable. In fact, the ideal sheaf $\mathcal{I}^2$ gives us a closed subscheme of $X\times_k X$ whose support is $\Delta_X(X)$ which is irreducible, but the quotient sheaf $ \cO_{X \times_k X}/\mathcal{I}^2$ is not reduced. This violates the hypothesis of \Cref{ideal_sheaf_isomorphism_lemma}. 
\end{rem}
\section*{Acknowledgment}
The authors would like to thank S\"{o}nke Rollenske for  carefully going through an earlier draft and pointing out subtle issues. Both the authors are supported by the INSPIRE faculty fellowship (Ref No.: IFA21-MA 161 and IFA22-MA 186) funded by the DST, Govt. of India.
\bibliographystyle{abbrv}
\bibliography{ref}

\begin{thebibliography}{10}

\bibitem{Berger}
R.~Berger.
\newblock Differentialmoduln eindimensionaler lokaler {R}inge.
\newblock {\em Math. Z.}, 81:326--354, 1963.

\bibitem{Biswas-Gurjar-Kolte}
I.~Biswas, R.~V. Gurjar, and S.~U. Kolte.
\newblock On the {Z}ariski-{L}ipman conjecture for normal algebraic surfaces.
\newblock {\em J. Lond. Math. Soc. (2)}, 90(1):270--286, 2014.

\bibitem{Flenner}
H.~Flenner.
\newblock Extendability of differential forms on nonisolated singularities.
\newblock {\em Invent. Math.}, 94(2):317--326, 1988.

\bibitem{Graf}
P.~Graf.
\newblock The generalized {L}ipman-{Z}ariski problem.
\newblock {\em Math. Ann.}, 362(1-2):241--264, 2015.

\bibitem{Graf-Kovacs}
P.~Graf and S.~J. Kov\'{a}cs.
\newblock An optimal extension theorem for 1-forms and the {L}ipman-{Z}ariski
  conjecture.
\newblock {\em Doc. Math.}, 19:815--830, 2014.

\bibitem{Greb-KP}
D.~Greb, S.~Kebekus, and T.~Peternell.
\newblock Reflexive differential forms on singular spaces. {G}eometry and
  cohomology.
\newblock {\em J. Reine Angew. Math.}, 697:57--89, 2014.

\bibitem{Greb-R}
D.~Greb and S.~Rollenske.
\newblock Torsion and cotorsion in the sheaf of {K}\"{a}hler differentials on
  some mild singularities.
\newblock {\em Math. Res. Lett.}, 18(6):1259--1269, 2011.

\bibitem{EGA-IV}
A.~Grothendieck.
\newblock \'{E}l\'{e}ments de g\'{e}om\'{e}trie alg\'{e}brique. {IV}. \'{E}tude
  locale des sch\'{e}mas et des morphismes de sch\'{e}mas {IV}.
\newblock {\em Inst. Hautes \'{E}tudes Sci. Publ. Math.}, (32):361, 1967.

\bibitem{Hartshorne_Algebraic}
R.~Hartshorne.
\newblock Algebraic de {R}ham cohomology.
\newblock {\em Manuscripta Math.}, 7:125--140, 1972.

\bibitem{Hartshorne_stable_reflexive_sheaves}
R.~Hartshorne.
\newblock Stable reflexive sheaves.
\newblock {\em Math. Ann.}, 254(2):121--176, 1980.

\bibitem{Huneke}
C.~Huneke, S.~Maitra, and V.~Mukundan.
\newblock Torsion in differentials and {B}erger's conjecture.
\newblock {\em Res. Math. Sci.}, 8(4):Paper No. 60, 15, 2021.

\bibitem{Jorder}
C.~J\"{o}rder.
\newblock A weak version of the {L}ipman-{Z}ariski conjecture.
\newblock {\em Math. Z.}, 278(3-4):893--899, 2014.

\bibitem{Kebekus}
S.~Kebekus.
\newblock Pull-back morphisms for reflexive differential forms.
\newblock {\em Adv. Math.}, 245:78--112, 2013.

\bibitem{Kunz}
E.~Kunz.
\newblock {\em K\"{a}hler differentials}.
\newblock Advanced Lectures in Mathematics. Friedr. Vieweg \& Sohn,
  Braunschweig, 1986.

\bibitem{Lipman}
J.~Lipman.
\newblock Free derivation modules on algebraic varieties.
\newblock {\em American Journal of Mathematics}, 87(4):874--898, 1965.

\bibitem{Miller-Vassiliadou}
C.~Miller and S.~Vassiliadou.
\newblock ({C}o)torsion of exterior powers of differentials over complete
  intersections.
\newblock {\em J. Singul.}, 19:131--162, 2019.

\bibitem{Red_book}
D.~Mumford.
\newblock {\em The red book of varieties and schemes}, volume 1358 of {\em
  Lecture Notes in Mathematics}.
\newblock Springer-Verlag, Berlin, expanded edition, 1999.
\newblock Includes the Michigan lectures (1974) on curves and their Jacobians,
  With contributions by Enrico Arbarello.

\bibitem{Mumford-Oda}
D.~Mumford and T.~Oda.
\newblock {\em Algebraic geometry. {II}}, volume~73 of {\em Texts and Readings
  in Mathematics}.
\newblock Hindustan Book Agency, New Delhi, 2015.

\bibitem{stacks-project}
T.~{Stacks Project Authors}.
\newblock \textit{Stacks Project}.
\newblock \url{https://stacks.math.columbia.edu}, 2018.

\end{thebibliography}
\end{document}